\documentclass[reqno,11pt]{amsproc}
\usepackage[canadian]{babel}
\usepackage{amssymb}
\usepackage{amsmath}
\usepackage{amsfonts}
\usepackage{microtype}
\usepackage{xcolor}
\usepackage[left=2.8cm,right=2.8cm]{geometry}
\usepackage{enumitem}
\usepackage{mathtools}

\definecolor{myurlcolor}{rgb}{0,0,0.4}
\definecolor{mycitecolor}{rgb}{0,0.5,0}
\definecolor{myrefcolor}{rgb}{0.5,0,0}
\usepackage[pagebackref,draft=false]{hyperref}
\hypersetup{colorlinks,
linkcolor=myrefcolor,
citecolor=mycitecolor,
urlcolor=myurlcolor}

\usepackage[capitalize]{cleveref}

\newcommand{\beq}{\begin{equation}}
\newcommand{\eeq}{\end{equation}}
\newcommand{\N}{\mathbb{N}}

\newcommand{\R}{\mathbb{R}}
\newcommand{\C}{\mathbb{C}}
\newcommand{\TR}{\mathbb{TR}}

\DeclareMathOperator{\diag}{\mathrm{diag}}
\DeclareMathOperator{\tr}{\mathrm{tr}}

\DeclareMathOperator{\SU}{SU}
\DeclareMathOperator{\GL}{GL}
\DeclareMathOperator{\SL}{SL}
\newcommand{\Rep}[1]{\mathsf{Rep}(#1)}		
\DeclareMathOperator{\Schur}{\mathsf{Schur}}	
\DeclareMathOperator{\chr}{\mathsf{ch}}	
\DeclareMathOperator{\WP}{\mathsf{WP}}	
\DeclareMathOperator{\interior}{int}
\DeclareMathOperator{\mult}{\mathsf{m}}

\swapnumbers
\newtheorem{dummy}{}[section]
\newtheorem{thm}[dummy]{Theorem}\Crefname{thm}{Theorem}{Theorems}
\newtheorem{lem}[dummy]{Lemma}\Crefname{lem}{Lemma}{Lemmas}
\newtheorem{prop}[dummy]{Proposition}\Crefname{prop}{Proposition}{Propositions}
\Crefname{cor}{Corollary}{Corollaries}

\newtheorem{defn}[dummy]{Definition}

\newtheorem{prob}[dummy]{Problem}
\theoremstyle{remark}
\Crefname{ex}{Example}{Examples}
\newtheorem{rem}[dummy]{Remark}

\numberwithin{equation}{section}

\let\originalleft\left
\let\originalright\right
\renewcommand{\left}{\mathopen{}\mathclose\bgroup\originalleft}
\renewcommand{\right}{\aftergroup\egroup\originalright}

\setlist[enumerate]{topsep=6pt,itemsep=6pt}
\setlist[itemize]{topsep=6pt,itemsep=6pt}

\allowdisplaybreaks

\newcommand{\newterm}[1]{\textbf{#1}}

\usepackage{todonotes}

\begin{document}
\sloppy

\setlength{\jot}{6pt}



\title{Asymptotic and catalytic containment\\ of representations of $\SU(n)$}

\author{Tobias Fritz}

\address{Department of Mathematics, University of Innsbruck, Austria}
\email{tobias.fritz@uibk.ac.at}



\thanks{\textit{Acknowledgements.} We thank Tim Netzer for useful discussions and Richard Stanley for pertinent input on MathOverflow, which has resulted in the proof of \Cref{real_schur} and indirectly in the proof of our main result, and an anonymous referee for helpful feedback.}

\begin{abstract}
	Given two finite-dimensional representations $\rho$ and $\sigma$ of $\SU(n)$, when is there $n \in \N$ such that $\rho^{\otimes n}$ is isomorphic to a subrepresentation of $\sigma^{\otimes n}$?
	When is there a third representation $\eta$ such that $\rho \otimes \eta$ is a subrepresentation of $\sigma \otimes \eta$?
	We call these the questions of \emph{asymptotic} and \emph{catalytic} containment, respectively.

	We answer both questions in terms of an explicit family of inequalities.
	These inequalities are almost necessary and sufficient in the following sense.
	If two representations satisfy all inequalities strictly, then asymptotic and catalytic containment follow (the former in generic cases).
	Conversely, if asymptotic or catalytic containment holds, then the inequalities must hold non-strictly.
	These results are an instance of a recent \emph{Vergleichsstellensatz} applied to the representation semiring. 
\end{abstract}

\maketitle

\tableofcontents

\section{Introduction}

The goal of this note is to prove the following result:

\begin{thm}
	\label{main}
	For $n \ge 2$ and finite-dimensional complex vector spaces $V$ and $W$, let
	\[
		\rho \: : \: \SU(n) \longrightarrow \GL(V), \qquad \sigma \: : \: \SU(n) \longrightarrow \GL(W)
	\]
	continuous representations of $\SU(n)$, or equivalently algebraic representations of $\SL(n,\C)$.
	Suppose that the following conditions hold:
	\begin{enumerate}
		\item\label{main_real} The associated characters $\chr$ satisfy
			\beq
				\label{main_real_ineq}
				\chr_\rho(x) \,<\, \chr_\sigma(x)
			\eeq
			for all $x \in \R_+^n$ with $x_1 \cdots x_n = 1$,
		\item\label{main_trop} For the associated weight polytopes\footnote{Recall that the weight polytope is defined as the convex hull of all weights.} $\WP$, we have
			\beq
				\label{main_trop_ineq}
				\WP(\rho) \,\subseteq\, \interior(\WP(\sigma)),
			\eeq
			where $\interior$ denotes interior.
	\end{enumerate}
	Then the following subrepresentation relations hold:
	\begin{enumerate}[label=(\alph*)]
		\item \newterm{Catalytic}: there is a nonzero representation $\eta$ such that
			\beq
				\label{main_cat}
				\rho \otimes \eta \hookrightarrow \sigma \otimes \eta.
			\eeq
		\item \newterm{Asymptotic}: if $\sigma$ is generic, then
			\beq
				\label{main_asymp}
				\rho^{\otimes n} \hookrightarrow \sigma^{\otimes n} \qquad \forall n \gg 1.
			\eeq
	\end{enumerate}
	Conversely, if \eqref{main_cat} holds for some nonzero $\eta$ or \eqref{main_asymp} holds for some $n \ge 1$, then the inequalities in~\ref{main_real} hold at least non-strictly, and similarly for~\ref{main_trop} in the sense that $\WP(\rho) \subseteq \WP(\sigma)$.
\end{thm}

In this theorem and throughout the paper, all of our representations are assumed finite-dimensional complex, and as usual continuous for $\SU(n)$ and algebraic for $\SL(n,\C)$, and we will use both of these interchangeably.
We say that a representation is \newterm{generic} if it contains both the trivial and the standard representation as subrepresentations.

The converse part mentioned at the end of \Cref{main} is easy to prove. We include it mainly in order to indicate that our forward direction, which gives a sufficient condition for catalytic and asymptotic containment, is almost necessary as well: the only difference between the two directions is the strictness of the inequalities (and the genericity assumption on $\sigma$ in the asymptotic case).

\begin{rem}
	Condition \ref{main_trop} can be thought of as extending~\ref{main_real} to infinity, in the sense that it is equivalent to
	\beq
		\label{main_trop_ineq2}
		\lim_{r \to \infty} \sqrt[r]{\chr_\rho(x^r)} \,<\, \lim_{r \to \infty} \sqrt[r]{\chr_\sigma(x^r)}
	\eeq
	for all $x \in \R_+^n$ with $x_1 \cdots x_n = 1$, where the map $x \mapsto x^r$ is defined componentwise.
\end{rem}

We prove \Cref{main} in \Cref{sec_proof}. 
It follows as an application of a recent \newterm{Vergleichsstellensatz} proven in~\cite{vergleichsstellensatz}, which is a separation theorem for preordered semirings that we recall in \Cref{vss_back}.
This application is facilitated by some preliminary considerations on Schur polynomials in~\Cref{sec_schur}.
In \Cref{SU2_case}, we present a somewhat more explicit form of \Cref{main} for $n = 2$ as \Cref{main_SU2}.

The main open problem is as follows.

\begin{prob}
	Formulate a generalization of \Cref{main} that holds similarly for all compact Lie groups in place of $\SU(n)$.
\end{prob}

Mainly, it is not clear to us on which set of points the corresponding inequalities~\eqref{main_real_ineq} will be required to hold in the case of a general compact Lie group.
We expect that everything else can be generalized verbatim, and also that the application of our Vergleichsstellensatz will still go through, based on suitably generalized versions of \Cref{real_schur,trop_schur}.

\section{Background: A Vergleichsstellensatz for preordered semirings}
\label{vss_back}

Here, we present our recent Vergleichsstellensatz~\cite[Theorem~7.15]{vergleichsstellensatz} for preordered semirings after a brief recap of the relevant definitions.
Recall that a \newterm{semiring} is like a ring, except in that the existence of additive inverses is not required. 
Throughout, all of our semirings are assumed commutative.

\begin{defn}
	A \newterm{preordered semiring} is a semiring $S$ together with a preorder relation $\le$ such that both addition and multiplication are monotone,
	\[
		a \le b \quad\Longrightarrow\quad a + c \le b + c, \quad ac \le bc.
	\]
\end{defn}

Two especially important preordered semirings are the following:
\begin{itemize}
	\item The nonnegative reals $\R_+$ with the standard ordering and algebraic structure.
	\item The \newterm{tropical reals} $\TR_+ \coloneqq (\R \cup \{-\infty\}, \max, +)$ with the standard ordering.
\end{itemize}

The following definition combines the notion of power universal element introduced as~\cite[Definition~3.27]{vergleichsstellensatz} with the special case characterization given in~\cite[Lemma~3.29]{vergleichsstellensatz}.

\begin{defn}
	Let $S$ be a preordered semiring with $1 \ge 0$. Then an element $u \in S$ is \newterm{power universal} if $u \ge 1$ and for every nonzero $a \in S$ there is $k \in \N$ such that
	\[
		a \le u^k, \qquad a u^k \ge 1.
	\]
\end{defn}

\begin{rem}
	\label{pu_rem}
	\begin{enumerate}
		\item If such $S$ has a power universal element and $1 \not\le 0$, then it follows that $S$ is both zerosumfree and has no zero divisors~\cite[Remark~3.36]{vergleichsstellensatz}.
		\item By $u \ge 1$, it is enough to check the power universality inequalities on a generating subset which includes the element $2 = 1 + 1$.
			This follows from the fact that the set of $a \in S$ that satisfy the relevant conditions is closed under addition (since it contains $2$) and under multiplication (obviously).
	\end{enumerate}
\end{rem}

Our Vergleichsstellensatz then reads as follows, where we restrict the present formulation to those statements that are of interest to us here.

\begin{thm}[{\cite[Theorem~7.15]{vergleichsstellensatz}}]
	\label{vss}
	Let $S$ be a preordered semiring with $1 > 0$ and a power universal element $u$. Suppose that nonzero $a, b \in S$ satisfy the following:
	\begin{enumerate}
		\item\label{real_spec} For all monotone homomorphisms $\phi : S \to \R_+$,
			\beq
				\label{real_spec_le}
				\phi(a) < \phi(b).
			\eeq
		\item\label{trop_spec} For all monotone homomorphisms $\psi : S \to \TR_+$ with $\psi(u) > 0$,
			\beq
				\label{trop_spec_le}
				\psi(a) < \psi(b).
			\eeq
	\end{enumerate}
	Then also the following hold:
	\begin{itemize}
		\item There is nonzero $c \in S$ such that
			\beq
				\label{cat_le}
				a c \le b c.
			\eeq
		\item If $b$ is power universal as well, then
			\beq
				\label{asymp_le}
				a^n \le b^n \qquad \forall n \gg 1.
			\eeq
	\end{itemize}
	Conversely, if \eqref{cat_le} holds for some nonzero $c$ or \eqref{asymp_le} holds for some $n \ge 1$, then the inequalities \eqref{real_spec_le} and \eqref{trop_spec_le} hold non-strictly.
\end{thm}

In~\cite{vergleichsstellensatz} we also gave a concrete form for a ``catalyst'' $c$ as in~\eqref{cat_le}, but this concrete form will be of less interest to us here.
A first application of a related Vergleichsstellensatz to the asymptotics of random walks on topological abelian groups has been given in~\cite{random_walks}.
The present paper presents our second application.

\section{The preordered semiring of Schur positive polynomials}
\label{sec_schur}

Fixing a number of variables $n \in \N$, we write $\Schur_n$ for the semiring of Schur positive symmetric integer polynomials in $n$ variables, preordered such that $p \le q$ if and only if $q - p$ is Schur positive.

Equivalently, $\Schur_n$ is the preordered semiring consisting of formal sums $\sum_\lambda p_\lambda s_\lambda$ of basis elements $s_\lambda$ indexed by partitions $\lambda = (\lambda_1 \ge \cdots \ge \lambda_\ell > 0)$ of length $\ell \le n$ and with coefficients $p_\lambda \in \N$. These multiply via the unique bilinear extension of
\beq
	\label{LR_mult}
	s_\mu s_\nu \,=\, \sum_\lambda c^\lambda_{\mu\nu} s_\lambda
\eeq
with Littlewood-Richardson coefficients $c^\lambda_{\mu\nu}$, and the preorder is the coefficientwise one.
The unit element of $\Schur_n$ is $1 = s_{()}$, the Schur polynomial associated to the empty partition.
For $0 \le j \le n$, we also write $e_j = s_{\underbrace{\scriptstyle{(1,\ldots,1)}}_{j\:\mathrm{times}}}$ for the corresponding elementary symmetric polynomial.

\begin{lem}
	\label{schur_bound}
	For every nonempty partition $\lambda$,
	\[
		s_\lambda \le e_1^{|\lambda|}.
	\]
\end{lem}

\begin{proof}
	By the Pieri rule, we have $c^\lambda_{(1)\, \nu} = 1$ if and only if $\lambda$ arises from $\nu$ by adding a single box, i.e.~if there is a row index $j$ such that $j = 0$ or $\nu_{j-1} > \nu_j$, and
	\[
		\lambda_i = \begin{cases}	\nu_i + 1 & \textrm{ if } i = j, \\
						\nu_i & \textrm{ if } i \neq j. \end{cases}
	\]
	We now prove the claim by induction on the size of $\lambda$.
	The statement is trivial for $\lambda = (1)$, which is the smallest nonempty partition, since $s_{(1)} = e_1$.
	For the induction step, let $\nu$ be such that $\lambda$ arises from $\nu$ by adding a single box as above.
	Then assuming $s_\nu \le e_1^{|\nu|}$ by induction hypothesis gives
	\[
		s_\lambda \le s_\nu e_1 \le e_1^{|\nu|+1},
	\]
	where the first inequality is by $c^\lambda_{(1)\, \nu} = 1$.
	This coincides with the claim by $|\lambda| = |\nu| + 1$.
\end{proof}

\begin{prop}
	\label{real_schur}
	The monotone homomorphisms $\Schur_n \to \R_+$ with trivial kernel can be parametrized by $y \in \R^n_{>0}$ and are given by the evaluation maps
	\begin{align*}
		\Schur_n & \longrightarrow \R_+, \\
		p & \longmapsto p(y).
	\end{align*}
\end{prop}

We thank Richard Stanley for the most difficult argument in the following proof, namely the final step of establishing $y_j \in \R$ based on~\cite{totallypositive}.\footnote{See also \href{https://mathoverflow.net/questions/419823/nonnegativity-locus-of-schur-polynomials}{https://mathoverflow.net/questions/419823/nonnegativity-locus-of-schur-polynomials}.}

\begin{proof}
	The fact that these maps take nonnegative values and are monotone is immediate by the sum over semistandard Young tablaux formula for the Schur polynomials $s_\lambda$, and it is obvious that they are homomorphisms. Also trivial kernel is obvious by $y_i > 0$ for all $i$.

	Conversely, let $\phi : \Schur_n \to \R_+$ be any monotone homomorphism with trivial kernel.
	With $\Lambda_n$ the ring of symmetric polynomials in $n$ variables, we have $\Lambda_n = \Schur_n - \Schur_n$ as the Schur polynomials form a basis.
	Therefore $\phi$ uniquely extends to a ring homomorphism $\Lambda_n \to \R$, which we also denote $\phi$ by abuse of notation, such that $\phi(s_\lambda) \ge 0$ for all partitions $\lambda$ of length $\ell(\lambda) \le n$.
	Consider now the monic single-variable real polynomial given by
	\[
		p(t) \coloneqq \sum_{i=0}^n \phi(e_i) \, t^{n-i} = \prod_{j=1}^n (t + y_j),
	\]
	where the $y_j$ are defined by the second equation, with the $-y_j$ being the roots of $p$.
	These are a priori complex numbers where the non-real ones come in complex conjugate pairs.
	However, we claim that actually $y_j \in \R_{>0}$ for all $j$, and that $\phi$ is given by the evaluation homomorphism at $y = (y_1,\ldots,y_n)$.
	Starting with the latter, we have $\phi(e_i) = e_i(y)$ for all $i$ by Vieta's formulas, and $\phi(f) = f(y)$ for all $f \in \Lambda_n$ then follows by the fundamental theorem on symmetric polynomials.

	It remains to be shown that $y_j \in \R_{>0}$.
	By a result of Aissen and Whitney~\cite[6.]{totallypositive}\footnote{See~\cite{totallypositiveI} for the proof of this result announced in~\cite{totallypositive}.}, our claim that $p$ has real and negative roots is equivalent to the Hankel matrix of its coefficients
	\[
		H \coloneqq \left( \phi(e_{i-j}) \right)_{i,j \in \N}
	\]
	being totally positive, where our notation uses padding by zeros for the elementary symmetric polynomials with index not in $\{0,\ldots,n\}$.
	But the nonzero minors of $H$ are all of the form $\phi(s_{\lambda / \mu})$ by the second Jacobi-Trudi formula~\cite[Corollary~7.16.2]{stanley2}, and the skew Schur polynomials $s_{\lambda / \mu}$ are well-known to be Schur positive, giving $\phi(s_{\lambda / \mu}) > 0$ by the trivial kernel assumption. It follows that $H$ is indeed totally positive.
\end{proof}

Although we will not need this, Vieta's formulas together with the fact that the coefficients of a polynomial uniquely determine its roots also implies that $y, y' \in \R_+^n$ define the same evaluation homomorphism if and only if they differ by a coordinate permutation.

\begin{prop}
	\label{trop_schur}
	The monotone homomorphisms $\Schur_n \to \TR_+$ with trivial kernel can be parametrized by $y \in \R^n$ and are given by \smallskip
	\begin{align}
		\begin{split}
			\label{trop_schur_eq}
			\Schur_n & \longrightarrow \TR_+, \\
			\sum_{\alpha \in \N^n} p_\alpha x^\alpha & \longmapsto \max_{\alpha \in \N^n \::\: p_\alpha \neq 0} \; \sum_{i=1}^n \alpha_i y_i.
		\end{split}
	\end{align}
\end{prop}

In other words, these homomorphisms are given by linear maximization over the Newton polytope of $p$ in a fixed direction specified by $y$.
This is similar to the case of polynomial semirings~\cite[Example~2.9]{vergleichsstellensatz}.

\begin{proof}
	We first argue that every such map, denote it $\psi_y$, is a monotone homomorphism with trivial kernel.
	Monotonicity of $\psi_y$ is obvious, since $p \le q$ implies that the Newton polytope of $p$ is contained in that of $q$.
	The fact that the Newton polytope of a product of two polynomials with nonnegative coefficients is the Minkowski sum of their Newton polytopes implies that $\psi_y$ is multiplicative.
	Additivity is a consequence of the Newton polytope of a sum being the convex hull of their unions.
	Hence $\psi_y$ is a monotone homomorphism, and it has trivial kernel since it takes the value $-\infty$ only on $p = 0$.

	Conversely, suppose that $\psi : \Schur_n \to \TR_+$ is a monotone homomorphism with trivial kernel.
	We first make some preliminary observations.
	\begin{itemize}
		\item The dual Pieri rule~\cite[p.~340]{stanley2}: the product of a Schur polynomial $s_\mu$ with an elementary symmetric polynomial $e_j$ decomposes as
			\beq
				\label{se_prod}
				s_\mu e_j = \sum_\lambda s_\lambda,
			\eeq
			where the sum is over all partitions $\lambda$ that can be obtained from $\mu$ by adding $j$ boxes in distinct rows.
		\item Applying $\psi$ to the multiplication rule~\eqref{LR_mult} yields
			\beq
				\label{psiadd}
				\psi(s_\mu) + \psi(s_\nu) \,=\, \max_{\lambda \: : \:  c^\lambda_{\mu\nu} \neq 0} \psi(s_\lambda)
			\eeq
			for all partitions $\mu$ and $\nu$, where throughout we only consider partitions of length $\le n$.
		\item $\psi$ is monotone with respect to dominance ordering:
			\[
				\mu \trianglelefteq \nu \quad \Longrightarrow \quad \psi(s_\mu) \le \psi(s_\nu).
			\]
			
			To see this, we can assume without loss of generality that $\mu$ is obtained from $\nu$ by moving a single box from row $i$ down to row $j > i$, meaning that $\mu$ and $\nu$ coincide in all rows apart from
			\[
				\mu_i = \nu_i - 1, \qquad \mu_j = \nu_j + 1,
			\]
			and $j = i + 1$ (box moves down by one row and left by any number of columns) or $\mu_i = \mu_j$ (box moves left by one column and down any number of rows)~\cite[Proposition~2.3]{brylawski}.
			Then any way of adding at least $m \coloneqq \sum_{k < \nu_i} (n - \nu'_k)$ boxes to $\mu$ so as to obtain a new partition necessarily also adds the extra box at $(i, \nu_i)$.
			By repeated application of the Pieri rule, it follows that every Schur polynomial that occurs in the expansion of $s_\mu e_1^m$ also occurs in the expansion of $s_\nu e_1^m$.
			Applying $\psi$ to this statement gives
			\[
				\psi(s_\mu) + m \psi(e_1) \,\le\, \psi(s_\nu) + m \psi(e_1),
			\]
			and hence also the desired $\psi(s_\mu) \le \psi(s_\nu)$.
		\item With $\mu + \nu$ denoting the usual sum of partitions, we have\footnote{We owe this observation, which facilitates a more insightful line of proof than our original one in the next item, to an anonymous referee.}
			\[
				\psi(s_{\mu + \nu}) = \psi(s_\mu) + \psi(s_\nu).
			\]

			Indeed this follows from the previous two items upon noting that $\mu + \nu$ is dominance maximal among all partitions $\lambda$ with $c^{\lambda}_{\mu\nu} \neq 0$ by the Littlewood-Richardson rule.
		\item For every $\lambda$,
			\beq
				\label{psi_se}
				\psi(s_\lambda) \,=\, \sum_{j=1}^{\lambda_1} \psi(e_{\lambda'_j}).
			\eeq

			This is a direct consequence of the previous item applied to the decomposition of $\lambda$ as a sum of columns.
		\item For all $j = 1, \ldots, n$,\footnote{We again use the conventions $e_0 = 1$ and $e_{n+1} = 0$.}
			\beq
				\label{submod}
				\psi(e_{j+1}) + \psi(e_{j-1}) \le 2 \psi(e_j).
			\eeq
			
			Indeed by~\eqref{se_prod}, every term in the Schur polynomial expansion of $e_{j-1} e_{j+1}$ also appears in the Schur polynomial expansion of $e_j e_j$. (In fact $e_j e_j = e_{j-1} e_{j+1} + s_{\scriptsize{(2,\ldots,2)}}$.)
	\end{itemize}
	To start the main argument, for every $f \in \Schur_n$ we have the unique decomposition into Schur polynomials
	\[
		f \,=\, \sum_\lambda f_\lambda s_\lambda
	\]
	with coefficients $f_\lambda \in \N$. Hence for the actual claim, for a given $\psi : \Schur_n \to \TR_+$ it is enough to show that there is $y \in \R^n$ such that
	\beq
		\label{psi_y}
		\psi(s_\lambda) \,=\, \psi_y(s_\lambda)
	\eeq
	for all $\lambda$. To this end, for every $j = 1,\ldots,n$ we put\footnote{Note that this is where the assumption that $\psi$ has trivial kernel comes in: it guarantees $\psi(e_{j-1}) > -\infty$.}
	\[
		y_j \,\coloneqq\, \psi(e_j) - \psi(e_{j-1}),
	\]
	where the second term vanishes for $j = 1$ since $e_0 = 1$.
	We then have $y_1 \ge \cdots \ge y_n$ thanks to~\eqref{submod}, and
	\beq
		\label{psiey}
		\psi(e_j) = \sum_{i=1}^j y_i
	\eeq
	by definition.
	By the sum over semistandard Young tablaux formula for $s_\lambda$, we can evaluate $\psi_y(s_\lambda)$ directly using the definition~\eqref{trop_schur_eq}.
	Because of $y_1 \ge \cdots \ge y_n$, the Young tablaux that attains the maximum is the minimal one, namely the one in which the $i$'th row is filled with only $i$'s for every $i = 1,\ldots,\ell(\lambda)$, so that
	\[
		\psi_y(s_\lambda) \,=\, \sum_{i=1}^{\ell(\lambda)} \lambda_i y_i.
	\]
	But then by the definition of the $y_i$ and~\eqref{psi_se}, we have
	\[
		\psi_y(s_\lambda) \,=\, \sum_{i=1}^{\ell(\lambda)} \lambda_i (\psi(e_i) - \psi(e_{i-1})) \,=\, \sum_{i=1}^{\ell(\lambda)} (\lambda_i - \lambda_{i+1}) \psi(e_i) \,=\, \sum_{j=1}^{\lambda_1} \psi(e_{\lambda'_j}) \,=\, \psi(s_\lambda),
	\]
	where the second to last equation is elementary and the last one is by~\eqref{psi_se}.
	This proves the desired~\eqref{psi_y}.
\end{proof}

\begin{rem}
	Our \Cref{vss} does not directly apply to $\Schur_n$ (or equivalently to the representation theory of $\GL(n,\C)$), since $\Schur_n$ does not have a power universal element: there is no $u$ such that $e_1 u^k \ge 1$ for any $k$, because expanding the left-hand side into Schur polynomials does not produce any constant terms.
	This is analogous to a polynomial semiring like $\R_+[x]$ not having a power universal element~\cite[Example~3.37]{vergleichsstellensatz}.
\end{rem}

\section{Proof of \texorpdfstring{\Cref{main}}{Theorem 1.1}}
\label{sec_proof}

We write $\Rep{\SL(n,\C)}$ for the set of isomorphism classes of representations, considered as a semiring with respect to direct sum as addition and tensor product as multiplication.
It is preordered with respect to containment of representations (up to isomorphism).

Next, we will identify $\Rep{\SL(n,\C)}$ with a preordered semiring constructed from $\Schur_n$.
We write
\[
	T_n \,\coloneqq\, \{ x \in \C^n \mid x_1 \cdots x_n = 1\} \,=\, \{ x \in \C^n \mid e_n(x) = 1\},
\]
and identify this variety with the maximal torus of $\SL(n,\C)$ consisting of the diagonal matrices.
Let now $\Schur_n^S$ be the semiring obtained by quotienting $\Schur_n$ by the semiring congruence generated by
\[
	e_n \sim 1.
\]
Under the isomorphism $\Rep{\GL(n,\C)} \cong \Schur_n$, this identifies the determinant representation with the trivial representation.

\begin{lem}
	Considering every $f \in \Schur_n^S$ as a function on $T_n$ identifies $\Schur_n^S$ with a subsemiring of the ring of regular functions on $T_n$.
\end{lem}

\begin{proof}
	Since $e_n = \prod_{i=1}^n x_i$ evaluates to $1$ on $T_n$, we have a well-defined homomorphism from $\Schur^S_n$ to the regular functions on $T_n$.
	To prove its injectivity, suppose that $f|_{T_n} = g|_{T_n}$ for representatives $f,g \in \Schur_n$.
	Then we have a symmetric polynomial $h \coloneqq f - g$ that vanishes on $T_n$.
	The fundamental theorem on symmetric polynomials now implies that $h$ is a multiple of the elementary symmetric polynomial $e_n = s_{(1,\ldots,1)}$.
\end{proof}

\begin{lem}
	$\Schur_n^S$ is a preordered semiring with $p \le q$ if and only if $q - p$ has a Schur positive representative.
\end{lem}

\begin{proof}
	For transitivity, suppose that $p \le q \le r$, so that both $q - p$ and $r - q$ have Schur positive representatives.
	Then their sum is a Schur positive representative of $r - p$, showing $p \le r$.
	All other relevant conditions are straightforward to verify.
\end{proof}

We henceforth consider $\Schur^S_n$ as a preordered semiring.
The following standard result then summarizes the well-understood representation theory in type $A_{n-1}$.

\begin{prop}
	\label{rep_iso}
	Associating to every representation $\rho : \SL(n,\C) \to \GL(V)$ the regular function on $T_n$ given by
	\beq
		\label{char_map}
		\chr_\rho(x_1,\ldots,x_n) \,\coloneqq\, \tr\left( \rho \left( \diag(x_1,\ldots,x_n) \right) \right)
	\eeq
	defines an isomorphism of preordered semirings $\Rep{\SU(n)} \stackrel{\cong}{\longrightarrow} \Schur_n^S$.
\end{prop}

Of course, the inverse isomorphism is given by decomposing a Schur positive polynomial $p$ into Schur polynomials, $p = \sum_\lambda p_\lambda s_\lambda$, constructing the direct sum of the associated irreducibles with multiplicities $p_\lambda$, and noting that this is well-defined on $\Schur_n^S$.

\begin{lem}
	\label{SchurS_pu}
	$u \coloneqq 1 + e_1$ is power universal in $\Schur_n^S$.
\end{lem}

This is equivalent to the standard representation-theoretic fact that every (not necessarily) irreducible representation is contained in a tensor power of the direct sum of the trivial and the standard representation.
We nevertheless offer a proof for completeness.

\begin{proof}
	By \Cref{pu_rem}, it is enough to verify the relevant inequalities $a \le u^k$ and $a u^k \ge 1$ on $a = 2$ and on $a = s_\lambda$ for any partition $\lambda$.
	Starting with the latter, in $\Schur_n$ we have $s_\lambda \le e_1^{|\lambda|}$ by \Cref{schur_bound}, and hence also $s_\lambda \le u^{|\lambda|}$ in $\Schur_n^S$. For the lower bound, we have $e_n^{\lambda_1} \le s_\lambda e_1^k$ with $k = \sum_{i=1}^n (\lambda_i - \lambda_1)$ by the dual Pieri rule, since every term in the Schur polynomial expansion on the left-hand side also appears on the right-hand side. This proves the relevant inequality in $\Schur_n^S$.

	Concerning $2$, note that $e_n \le e_1^n$ gives the upper bound $2 = 1 + e_n \le u^n$ in $\Schur_n^S$.
	The lower bound $2 \ge 1$ is trivial.
\end{proof}

The main part of the proof is now straightforward boilerplate:

\begin{proof}[Proof of \Cref{main}]
	Using \Cref{rep_iso}, we work with $\Schur^S_n$ instead of $\Rep{SU(n)}$.	
	By \Cref{SchurS_pu}, we know that every element that corresponds to a generic representation is power universal.

	The proof is then completed by the following observations:
	\begin{itemize}
		\item The monotone homomorphisms $\Schur_n^S \to \R_+$ correspond to the evaluation maps at $x \in T_n$.
			Indeed this is an immediate consequence of \Cref{real_schur} together with the definition of $\Schur_n^S$ as the quotient of $\Schur_n$ by $e_n \sim 1$, and noting that every such homomorphism necessarily has trivial kernel by the existence of a power universal element.
		\item \Cref{real_schur} tells us that the monotone homomorphisms $\Schur_n \to \TR_+$ are given by linear optimization over the Newton polytope in a fixed direction $y \in \R^n$.
			Such a homomorphism descends to $\Schur_n^S$ if and only if $\sum_i y_i = 0$.
			Under the isomorphism $\Schur_n^S \cong \Rep{\SL(n,\C)}$, these homomorphisms correspond exactly to the linear optimization over the weight polytope.
			This produces condition~\ref{main_trop} upon noting that linear optimization over one polytope is strictly dominated by linear optimization over a second one in all directions if and only if the first is contained in the interior of the second.
			\qedhere
	\end{itemize}
\end{proof}

\section{The case of \texorpdfstring{$\SL(2,\C)$}{SL(2,\C)}}
\label{SU2_case}

It is worth considering the case $n = 2$ in a bit more detail, since there the conditions of \Cref{main} take a more explicit form.
We identify representations of $\SU(2)$ with their multiplicities.
These are maps $\mult : \N_{>0} \to \N$, assigning to each $d > 0$ the multiplicity of the irreducible representation of dimension $d$.

\begin{thm}
	\label{main_SU2}
	Let
	\[
		\rho : \SL(2,\C) \to \GL(V), \qquad \sigma : \SL(2,\C) \to \GL(W)
	\]
	be algebraic representations with multiplicity maps $\mult_\rho, \, \mult_\sigma : \N_{>0} \to \N$.
	Suppose that the following conditions hold:
	\begin{enumerate}[label=(\roman*')]
		\item\label{SU2_real} For all $\alpha \in [0,\infty)$,
			\beq
				\label{SU2_real_ineq}
				\sum_{d \,\in\, \N_{>0}} \mult_\rho(d) \, \frac{\sinh(\alpha d)}{\sinh(\alpha)}
				\,<\,
				\sum_{d \,\in\, \N_{>0}} \mult_\sigma(d) \, \frac{\sinh(\alpha d)}{\sinh(\alpha)}
			\eeq
		\item\label{SU2_trop}
			\beq
				\label{SU2_trop_ineq}
				\max \{ d \in \N_{>0} \mid \mult_\rho(d) > 0 \} \,<\, 
				\max \{ d \in \N_{>0} \mid \mult_\sigma(d) > 0 \}.
			\eeq
	\end{enumerate}
	Then the following subrepresentation relations hold:
	\begin{enumerate}[label=(\alph*)]
		\item \newterm{Catalytic}: there is a nonzero representation $\eta$ such that
			\beq
				\label{SU2_cat}
				\rho \otimes \eta \hookrightarrow \sigma \otimes \eta.
			\eeq
		\item \newterm{Asymptotic}: if $\sigma$ is generic, then
			\beq
				\label{SU2_asymp}
				\rho^{\otimes n} \hookrightarrow \sigma^{\otimes n} \qquad \forall n \gg 1.
			\eeq
	\end{enumerate}
	Conversely, if \eqref{SU2_cat} holds for some nonzero $\eta$ or \eqref{SU2_asymp} holds for some $n \ge 1$, then the inequalities in~\eqref{SU2_real_ineq} and~\eqref{SU2_trop_ineq} hold at least non-strictly.
\end{thm}

For $\alpha = 0$, the inequality~\eqref{SU2_real_ineq} needs to be understood in terms of L'Hôpital's rule, which makes it equivalent to
\[
	\dim(\rho) \,<\, \dim(\sigma).
\]
In its non-strict form, this is an obvious necessary condition for catalytic and asymptotic containment.

\begin{proof}
	In order to derive this from \Cref{main}, it is enough to identify condition \ref{main_real} there with the present \ref{SU2_real}, and similarly for \ref{main_trop} with \ref{SU2_trop}.

	For the former, let $\iota_d : \SL(2,\C) \to \GL(\C^d)$ be the irrep of dimension $d \ge 1$.
	It is well-known that its character satisfies
	\[
		\tr(\iota_d(\diag(e^{\alpha}, e^{-\alpha})) = \frac{\sinh(\alpha d)}{\sinh(\alpha)},
	\]
	and the L'Hôpital extension gives the correct character value at $\alpha = 0$, namely the dimension $d$.
	This implies that the two sides of~\eqref{SU2_real_ineq} are the relevant character values for $\rho$ and $\sigma$.
	Since every solution of $x_1 x_2 = 1$ with $x_1, x_2 \in \R_+$ is of the form $x_1 = e^\alpha$ and $x_2 = e^{-\alpha}$ for some $\alpha \in \R$, the equivalence with \ref{main_real} follows upon noting that we can assume $\alpha \ge 0$ without loss of generality.

	For the tropical part, note that the weight polytope of the irrep $\iota_d$ can be identified with the interval $[-(d-1),+(d-1)] \subseteq \R$.
	Since the weight polytope of a general representation is the union of the weight polytopes of its irreducible components,
	the weight polytope criterion~\ref{main_trop} from \Cref{main} is manifestly equivalent to the present condition~\ref{SU2_trop}.
\end{proof}

\begin{rem}
	Our proof of \Cref{main} relied on the classification of monotone homomorphisms $\Rep{\SL(n,\C)}\to\R_+$, which we have performed in the proof of \Cref{main} based on \Cref{real_schur}.
	For $n = 2$, the same classification formulated in terms of the hyperbolic sine as in~\eqref{SU2_real_ineq} also appears in work of Sz\'ekelyhidi and Vajday on the $\SU(2)$ hypergroup~\cite[Theorem~1]{SV}.
\end{rem}

\bibliographystyle{plain}
\bibliography{asymptotic_reps}

\end{document}